\documentclass[12pt]{amsart}
\usepackage{amsmath,latexsym,amsfonts,amssymb,amsthm}
\usepackage{tikz-cd}

\usepackage{geometry}
\usepackage{hyperref}
\usepackage{mathrsfs}
\usepackage{graphicx,color}
\usepackage{tikz-cd}
\usepackage{tikz}
\usetikzlibrary{positioning}
\usepackage{enumerate}
\usepackage{mathtools}
\usepackage[all,cmtip]{xy}

\newtheorem{prop}{Proposition}[section]
\newtheorem{thm}[prop]{Theorem}
\newtheorem{cor}[prop]{Corollary}

\newtheorem{lem}[prop]{Lemma}

\theoremstyle{definition}

\newtheorem{defn}[prop]{Definition}

\newtheorem{remark}[prop]{\it Remark}

\newtheorem*{claim*}{Claim}

\DeclareMathOperator{\Spec}{Spec}

\DeclareMathOperator{\Pic}{Pic}
\DeclareMathOperator{\Cl}{Cl}
\DeclareMathOperator{\ct}{ct}
\DeclareMathOperator{\mct}{mct}

\DeclareMathOperator{\ord}{ord}

\DeclareMathOperator{\sdeg}{s-deg}
\DeclareMathOperator{\Bir}{Bir}

\newcommand{\dashto}{\dashrightarrow}

\begin{document}

\title{Noether-Fano Inequalities and Canonical Thresholds on Fano Varieties}

\author{Charlie Stibitz}
\address{Department of Mathematics, Northwestern University, Evanston, IL}
\email{charles.stibitz@northwestern.edu}

\maketitle

\begin{abstract}
    We prove a more general and precise version of the Noether-Fano inequalities for birational maps between Mori fiber spaces. This is applied to give descriptions of global canonical thresholds on Fano varieties of Picard number one. 
\end{abstract}

\section{Introduction}

\indent The Noether-Fano inequalities have a history going back to the beginnings of birational geometry. Intuitively they say that birational maps between Mori fiber spaces create singular divisors on the two Mori fiber spaces. In their modern form, where one of the varieties is a Fano variety of a Picard number 1, they can be stated as follows using singularities of pairs.

\begin{prop}[Noether-Fano Inequalities \cite{corti1995}] \label{NFinequalitiesCorti}
Suppose that $f:X \dashto Y$ is a birational map which is not an isomorphism from a terminal Fano variety $X$ with $\rho(X) = 1$ to a terminal Mori fiber space $q:Y \to S$. Let $H_{Y}$ be a general element of a very ample linear system  on $Y$ and denote by $H_{X} := f_{\ast}^{-1}H_{Y} \sim -mK_{X}$. Then the pair $\left(X,\frac{1}{m}H_{X}\right)$ is not canonical.
\end{prop}

\indent In fact by \cite{CS08}, the implication can be reversed and from any movable linear system $\mathcal{M} \sim -mK_{X}$ on $X$ we can construct a birational map to a Mori fiber space assuming that $\left(X,\frac{1}{m}\mathcal{M}\right)$ is not canonical. On the other hand, the Noether-Fano inequalities above have no chance of being sharp as they only say $\ct\left(X,\frac{1}{m}\mathcal{M}\right) < 1$. In this paper, we introduce a more precise version of the above inequalities that are in fact sharp.

\begin{thm}\label{NFInequalities}
Suppose that $f:X \dashto Y$ is a birational from a terminal Fano variety of Picard number $1$ to a Mori fiber space $q:Y \to S$. Denote by $d = \sdeg(f)$ and $d' = \sdeg(f^{-1})$ the Sarkisov degrees of $f$ and $f^{-1}$ (see Definition \ref{sdeg}). Then if $D\sim -mK_{X}$ is a divisor contracted by $f$
\[\ct\left(X,\frac{1}{m}D\right) \leq \frac{d'-1}{d'}.\]
Suppose instead $\mathcal{M} = f_{\ast}^{-1}\mathcal{H} \sim -mK_{X}$ is the strict transform of a very ample linear system $\mathcal{H} \sim_{\mathbb{Q}} -\ell K_{Y}$ on $Y$. Then  
\[\ct\left(X,\frac{1}{m}\mathcal{M}\right) \leq \frac{d'-1}{d'-1/d}.\]
\end{thm}

\indent This inequality says that birational maps to Mori fiber spaces lead to singular divisors on $X$ measured by the canonical threshold. In Proposition \ref{computecanonicalthresholds}, we show that given a pair $\left(X,\frac{1}{m}\mathcal{M}\right)$ with $\mathcal{M} \sim -mK_{X}$ these inequalities actually become an equality for the right choice of $f$, assuming that $\ct\left(X,\frac{1}{m}\mathcal{M}\right) < 1$. This naturally leads to the study of global canonical thresholds.

\begin{defn}
Suppose that $X$ is a terminal Fano variety of Picard number one. The \textit{global canonical threshold} is defined to be: \[\ct(X) := \sup\left\{c \mid \left(X,\frac{c}{m}\mathcal{D}\right) \text{ is canonical for every Weil divisor $D\sim -mK_{X}$} \right\}.\]
Similarly the \textit{global movable canonical threshold} is defined to be: 
\[\mct(X) := \sup\left\{c \mid \left(X,\frac{c}{m}\mathcal{M}\right) \text{ is canonical for every movable linear system $\mathcal{M}\sim -mK_{X}$} \right\}.\]
\end{defn}

\indent If we replace canonical thresholds with log canonical thresholds, these invariants are better known. The global log canonical threshold is equivalent to Tian's alpha invariant used to show the existence of K\"{a}hler-Einstein metrics on Fano varieties \cite{tian87}. The movable version and its relation to $K$-stability has been studied in \cite{stibitzzhuang} and \cite{zhu2020higher}. \\
\indent For global canonical thresholds the Noether-Fano inequalities above say that birational superrigidity is equivalent to proving $\mct(X) \geq 1$. Less studied is the invariant $\ct(X)$, though there is a result of Pukhlikov showing that $\ct(X) = 1$ for certain general complete intersections of index $1$ \cite{pukhlikov2018canonical}. In this paper, we give a geometric description of these invariants in terms of the Sarkisov degrees to other Mori fiber spaces. To do so we need to define the following classes of birational maps out of $X$.

\begin{defn}
Given a terminal Fano variety $X$ of Picard number 1, we define the following sets of birational maps from $X$.
\begin{enumerate}
    \item $\mathscr{M} = \{f:X \dashto Y\}$ where $q:Y \to S$ is a Mori fiber space with $\dim(S) \geq 1$.
    \item $\mathscr{F} = \{f:X \dashto Y\}$ where $Y$ is a Fano variety of Picard number 1. 
    \item  $\mathscr{S}_{1} = \{f:X \dashto Y\}$ where $f$ is a Sarkisov link of type I to a terminal Mori fiber space $q:Y \to S$ with $\dim(S) \geq 1$.
    \item $\mathscr{S}_{2} = \{f:X \dashto Y\}$ where $f$ is a Sarkisov link of type II to a terminal Fano variety $Y$ of Picard number 1.
    \item $\mathscr{S}_{2}^{+} = \{f:X \dashto Y\}$ where $f$ is a Sarkisov-like link of type II (see Definition \ref{Slinks}) to a klt Fano variety $Y$ of Picard number 1.
\end{enumerate}

\end{defn}

\indent Using these classes of maps, the Noether-Fano inequalities plus their sharpness imply the following theorem.

\begin{thm}\label{GlobalCanonicalThresholds}
Suppose that $X$ is a terminal Fano variety of Picard number $1$. For a birational map $f:X \dashto Y$ we will write $d = \deg(f)$ and $d' = \deg(f^{-1})$. Suppose that $\ct(X) < 1$ then 
\[\ct(X) = \inf\left\{\frac{d'-1}{d'}|f \in \mathscr{M},\mathscr{F}\right\} = \inf\left\{\frac{d'-1}{d'}|f \in \mathscr{S}_{1},\mathscr{S}_{2},\mathscr{S}_{2}^{+}\right\}.\]
Moreover if $\mct(X) < 1$ as well, then
\begin{align*}
    \mct(X) & =\inf\left\{\frac{d'-1}{d'}|f \in \mathscr{M}\right\}\cup \left\{\frac{d'-1}{d'-1/d}|f \in \mathscr{F}\right\} \\
&= \inf\left\{\frac{d'-1}{d'}|f \in \mathscr{S}_{1}\right\}\cup \left\{\frac{d'-1}{d'-1/d}|f \in \mathscr{S}_{2}\right\}.
\end{align*}
\end{thm}

\begin{remark}
It is an interesting question to ask if the above infima are in fact minima, or in other words, are the global canonical thresholds computed by a divisor or movable linear system. For the global log canonical threshold this is known to be true by \cite{birkar2016singularities}. On the other hand, for global movable log canonical thresholds  this is still unknown \cite{zhu2020higher}. In the case of canonical thresholds it does not seem to follow immediately from Theorem \ref{GlobalCanonicalThresholds} and boundedness of terminal Fano varieties. For $\ct(X)$ we need to look at $\sdeg(f^{-1})$ where $f\in \mathscr{S}_{2}^{+}$ is a map to a klt Fano variety and in particular there is no way to give a lower bound on the degree. For $\mct(X)$ although $\sdeg(f^{-1})$ can be bounded, it is in theory possible that $\sdeg(f)$ does not have an upper bound. 
\end{remark}

\indent As a first application, we show how to compute $\mct(X)$ for birationally rigid varieties. If all the Sarkisov links out of $X$ are known this is a simple calculation from Theorem \ref{GlobalCanonicalThresholds}. Hence this same computation applies similarly to any variety where all the degrees of the links out are explicitly known. 

\begin{cor}\label{mctrigid}
Suppose that $X$ is a birationally rigid Fano variety with $\Pic(X) = \Cl(X) = \mathbb{Z}\cdot -K_{X}$. Then $\mct(X) > \frac{1}{2}$. Moreover if $\Bir(X)$ is generated by involutions of Sarkisov degree $d$ then $\mct(X) = \frac{d}{d+1} \geq \frac{2}{3}$. 
\end{cor}

\indent It would be interesting to try to use the links to give lower bounds for $\ct(X)$ as well. This would in particular also give bounds for $\alpha(X)$. Unfortunately, there does not seem to be a clear way to use birational geometry to restrict the class $\mathscr{S}_{2}^{+}$, of Sarkisov-like links of type II out of a variety $X$. Even for birationally superrigid Fano varieties, this is an interesting question. For instance, proving that $\sdeg(f^{-1}) \geq \frac{1}{2}$ for any birational map $f:X \dashto Y$, where $X$ is birationally superrigid and $f \in \mathscr{S}_{2}^{+}$, would prove the K-stability of $X$ \cite{stibitzzhuang}. \\
\indent In the opposite direction though, we can use conclusions about $\ct(X)$ to restrict larger classes of birational maps out of $X$. For instance, we can look at Fano varieties of Picard number $1$ with $\ct(X) = 1$. Such Fano varieties are birationally superrigid, and therefore have no birational maps to terminal Fano varieties of Picard number one or Mori fiber spaces with positive dimensional bases. The following says that they also do not have any birational map to Fano varieties of Picard number 1 with worse singularities, so they have no birational map to any Mori fiber space no matter the singularities. Such Fano varieties are known to exist. For example, in \cite{pukhlikov2018canonical}, general complete intersections of index 1 with certain restrictions on the degree and codimension are shown to satisfy $\ct(X) = 1$.

\begin{cor}\label{ct1}
Suppose that $X$ is a terminal Fano variety of Picard number $1$ such that $\ct(X) = 1$. Then there exists no birational map $f:X \dashto Y$ where $Y$ is a Fano variety of Picard number 1, regardless of the singularities on $Y$. 
\end{cor}

\indent In a similar direction, we can prove something for a larger class of birationally superrigid varieties with a weaker conclusion.

\begin{cor}\label{bsrmaps}
Suppose that $X$ is  birationally superrigid Fano variety with $\Pic(X) = \Cl(X) = \mathbb{Z}\cdot -K_{X}$. Suppose that $Y$ is another Fano variety of Picard number 1 with a base-point-free linear system $\mathcal{H}_{Y} \sim_{\mathbb{Q}} -\ell K_{Y}$ with $\ell < 1$. Then there exists no birational map $f:X \dashto Y$.
\end{cor}

\indent In particular, the above corollary can be applied to the case when $Y$ is any Fano hypersurface of index $> 2$. Then the corollary says that a birationally superrigid variety has no birational map to any Fano hypersurface of index $\geq 2$ no matter the singularities. \\
\indent The applications and form of the Noether-Fano inequalities has evolved steadily over time. In 1869 Noether claimed that birational automorphisms of $\mathbb{P}^{2}$ factor as a product of quadratic one. He claimed this follows from the inequality
\[\nu_{1} + \nu_{2} + \nu_{3} \geq d,\]
where $\nu_{i}$ denote multiplicities of the strict transform of a general line under the map (including infinitely near multiplicities) and $d$ denotes the degree of the map. Although Noether proved the inequality, the proof of the factorization of Cremona transformation into quadratic ones 
 was not completely settled until Castelnuovo. Moreover, Castelnuovo proved that any birational map between either $\mathbb{P}^{2}$'s or ruled surfaces factors into a sequence of blow ups $\mathbb{P}^{2} \dashto \mathbb{F}_{1} $, elementary transformations $\mathbb{F}_{k} \dashto \mathbb{F}_{k\pm 1}$, contractions $\mathbb{F}_{1} \to \mathbb{P}^{2}$, and involutions $\mathbb{F}_{0} \to \mathbb{F}_{0}$. \\
\indent In dimension three, Fano in 1915 made the first attempt to generalize Noether's inequality. He correctly proved forms of the inequalities, where there are multiple cases for the multiplicity to be considered. He applied this to study birational maps of what are now termed Fano threefolds, though there were some flaws in his methods. In 1971, Iskovskikh and Manin \cite{IM71} fixed the flaws in Fano's argument and were able to prove that any birational map between smooth quartic threefolds is an isomorphism. \\
\indent The factorization theorem of Casteluovo was also extended to higher dimensions. Sarkisov and Reid proposed a way to factorize birational maps between threefold Mori fiber spaces. Using the Noether-Fano inequalities and a degree related to the Sarkisov degree used here, Corti was able to prove the factorization in dimension 3 \cite{corti1995}. This also lead to the modern form of Noether-Fano inequalities as in Proposition \ref{NFinequalitiesCorti}. In higher dimensions, the Sarkisov program is now known by \cite{hmsarkisov}, though the methods are very different. \\
\indent In another direction, the results of Iskovskikh and Manin have been generalized to many other varieties as well. Thisled to the definition of a variety being birationally superrigid, meaning it has no birational map to a Mori fiber space which is not an isomorphism. Many index one Fano varieties of Picard number one are now known to be birationally superrigid. In particular, all smooth index one hypersurfaces $X \subseteq \mathbb{P}^{N}$ are superrigid by \cite{df2013}. Outside of index $1$ not much is known except in some special cases (e.g. \cite{pukhlikovindex2}). For a more complete summary of the history of the Noether-Fano inequalities see \cite{iskovskikh2004noether}. 

\subsection*{Acknowledgements}
The author would like to thank J\'{a}nos Koll\'{a}r and Mihnea Popa for helpful comments on this paper.

\section{The Noether-Fano Inequalities}

Throughout this section, we will consider a birational map $f:X \dashto Y$ of normal $\mathbb{Q}$-factorial varieties over a algebraically closed field $k$ of characteristic $0$. We will moreover assume that each of the varieties $X$ and $Y$ has a projective fibration, denoted by $p:X \to T$ and $q:Y \to S$, such that $\rho(X/T) = \rho(Y/S) = 1$, with $-K_{X}$ (resp. $-K_{Y}$) generating the relative Picard group $\Pic(X/T)\otimes \mathbb{Q}$ (resp. $\Pic(Y/S) \otimes \mathbb{Q}$). Although our main focus will be when $X \to T$ and $Y \to S$ are Mori fiber spaces, we will be able to consider the inequalities in the following generalizations: $X$ and $Y$ may have worse than terminal (or klt) singularities, $K_{X}$ and $K_{Y}$ may be relatively ample, and the maps $p,q$ may be birational. It is also possible to consider the case where $-K_{X}$ or $-K_{Y}$ is relatively trivial, though we must choose a generator of the relative Picard group. Our main example will be the following situation. 

\begin{defn}
A Mori fiber space $p:X \to T$ is a dominant morphism of normal projective varieties where
\begin{itemize}
    \item $p_{\ast}\mathcal{O}_{X} = \mathcal{O}_{T}$ and $\dim(T) < \dim(X)$;
    \item $X$ is $\mathbb{Q}$-factorial;
    \item $\rho(X/T) =1$;
    \item $-K_{X}$ is $p$-ample. 
\end{itemize}
\end{defn}
\indent Usually an assumption on the singularities of $X$ is also made (e.g. $X$ is either terminal or klt). We will specify the singularities these singularities as needed. Moreover some of the results here  will apply with no assumptions at all on the singularities of $X$. \\
\indent Going back to the general setup, consider a resolution $\Gamma$ of the graph of $f$ and denote by $\pi:\Gamma \to X$ and $\rho:\Gamma \to Y$ the projections. Let $E_{i}$ be the collection $\pi$-exceptional divisors and $F_{j}$ be the collection of $\rho$ exceptional divisors. We define numbers $e_{i}$, $f_{j}$ by the linear equivalence $\rho_{\ast}E_{i} \sim_{q,\mathbb{Q}} -e_{i}K_{Y}$ and $\pi_{\ast}F_{j} \sim_{q,\mathbb{Q}} -f_{j}K_{X}$. \\
\indent We now define the notion of degree for the birational maps, $f:X \dashto Y$, considered above. We use the name Sarkisov degree based on the similar definition in \cite{corti1995}.

\begin{defn} \label{sdeg} Suppose that $f:X \dashto Y$ is a birational map of   normal varieties as above.  Then we define the \textit{Sarkisov degree} of $f$, denoted by $\sdeg(f)$ as follows. Consider general elements $A,B$ of very ample linear systems on $Y$ such that $A-B \sim_{\mathbb{Q}} -\ell K_{Y}$ with $\ell \neq 0$. Then we can write the strict transform $f_{\ast}^{-1}(A-B) \sim_{\mathbb{Q},p} -m K_{X}$. Define 
\[\sdeg(f) := \frac{m}{\ell} \in \mathbb{Q.}\]
\end{defn}

\indent If $-K_{Y}$ is semiample the definition simplifies: we may take a general element of $|-\ell K_{Y}|$ for $m >> 0$ and consider its strict transform which is linearly equivalent to $-m K_{X}$ over $T$.  

\begin{remark}
When $X = Y = \mathbb{P}^{N}$ and $S,T = \Spec(k)$, $\sdeg(f)$ is equal to the usual degree of the Cremona transformation $f$. In general, this is based of the definition in \cite{corti1995} where a related invariant also termed the Sarkisov degree is used to prove any birational map of threefold Mori fiber spaces factors into special links. This version of the Sarkisov degree differs in two ways from that of Corti. First, when $-K_{Y}$ is not ample, Corti fixes a very ample linear system $\mathcal{H} \sim_{q} -\ell K_{Y}$, and uses its strict transform to define the degree. Second, Corti's definition of Sarkisov degree consists of two extra invariants that are used to prove the Sarkisov program in dimension 3.  
\end{remark}

\indent We now define a way to measure the singularities of divisors on $X$ (or $Y$). This takes the form of a weighted sum of valuations on $X$ determined by the divisors $E_{i}$ and weights $e_{i}$ (or $F_{j}$ and $f_{j}$). 

\begin{defn} [Weighted Valuation associated to a birational map] Suppose that $f:X \dashto Y$ is a birational map of varieties as above. Define the \textit{weighted valuation} $\nu_{f}$ associated to $f$ to be the function
\[\nu_{f} = \sum_{i}e_{i}\ord_{E_{i}},\]
where $E_{i}$ ranges over the $\pi:\Gamma \to X$ exceptional divisors and $\rho_{\ast}E_{i} \sim_{q,\mathbb{Q}} -e_{i}K_{Y}$
\end{defn}

\indent This weighted valuation is not necessarily a valuation as it does not satisfy $\nu_{f}(a+b) \geq \min(\nu_{f}(a),\nu_{f}(b))$. In fact, it is just a homomorphism from $K(X)^{\times} \to \mathbb{R}$. Nevertheless it will be useful to think of it as a weighted average  of the valuations $\ord_{E_{i}}$. With this in mind, we define a weighted discrepancy as follows.

\begin{defn} [Discrepancy of weighted valuation] Suppose that $\nu = \sum_{i}e_{i}\ord_{E_{i}}(\cdot)$ is a weighted valuation, where $e_{i} \in \mathbb{R}$, on a variety $X$ such that $K_{X}$ is $\mathbb{Q}$-Cartier. Then we define the discrepancy $a(\nu)$ of $\nu$ to be the sum 
\[a(\nu) := \sum_{i}e_{i}a(E_{i},X).\]
\end{defn}

\indent We will now show how $\sdeg(f^{-1})$ relates to the discrepancy of $\nu_{f}$. 

\begin{prop} \label{degformula}
Suppose that $f:X \dashto Y$ is a birational map as above. Then 
\[
\sdeg(f^{-1}) = 1 + a(\nu_{f})
\]
\end{prop}
\begin{proof}
Consider general elements of very ample linear systems $A,B$ such that $A-B \sim -mK_{X}$ for some $m$. Then the strict transform $A_{\Gamma}-B_{\Gamma} := \pi^{-1}_{\ast}(A-B) \sim -m\pi^{\ast}K_{X} \sim -mK_{\Gamma} + \sum_{i}ma_{i}E_{i} $ where $a_{i} = a(E_{i},X)$. 
Pushing forward to $Y$ gives 
\[f_{\ast}(A-B) \sim \rho_{\ast}(A_{\Gamma}-B_{\Gamma}) \sim -mK_{Y} + \sum_{i}ma_{i}\rho_{\ast}E_{i} \sim_{q,\mathbb{Q}} -m\left(1 + \sum_{i}a_{i}e_{i}\right)K_{Y} \]
\end{proof}

\indent In a similar manner, given any divisor $D\sim_{\mathbb{Q}} -mK_{X}$ we can determine what multiple of $-K_{Y}$ the strict transform, $f_{\ast}D$ is linearly equivalent to.

\begin{prop} \label{Pushforward}
Suppose that $D \sim_{\mathbb{Q}} -mK_{X}$ is any Weil divisor on $X$. Then 
\[f_{\ast}D \sim_{q,\mathbb{Q}} -m\left(\sdeg(f^{-1}) - \nu_{f}\left(\frac{1}{m}D\right)\right)K_{Y}\]
\end{prop}
\begin{proof}
Let $b_{i} = \ord_{E_{i}}(D)$. Then the strict transform $D_{\Gamma}$ of $D$ on $\Gamma$ satisfies
\[D_{\Gamma}\sim_{\mathbb{Q}} \pi^{\ast}D - \sum b_{i}E_{i} \sim_{\mathbb{Q}} - mK_{X} + \sum_{i}(m a_{i}-b_{i})E_{i}.\] 
Pushing this divisor forward to $Y$ gives
\[f_{\ast}D \sim_{q,\mathbb{Q}} -mK_{Y} +\sum_{i}e_{i}(ma_{i}-b_{i})K_{Y} \sim_{q,\mathbb{Q}} -m (\sdeg(f^{-1}) - \nu_{f}(D/m))K_{Y}.\]
\end{proof}

\indent In the following important case, we can prove that $\sdeg(f^{-1}) > 1$.

\begin{prop}\label{degbound}
Suppose that $f:X \dashto Y$ is a birational map where $X$ is a terminal Fano variety with $\rho(X) = 1$ and  $q:Y \to S$ is any Mori fiber space. Then $\sdeg(f^{-1}) > 1$.
\end{prop}
\begin{proof}
This follows from Proposition \ref{degformula}. As $X$ is terminal $a(E_{i},X) > 0$ for each $E_{i}$. Moreover as $Y$ is a Mori fiber space $e_{i} \geq 0$. As $\rho(Y) \geq 2$, all the $e_{i}$ cannot be $0$. 
\end{proof}

\indent Using the pushforward formula above we can prove the following theorem which can be interpreted as a general form of the Noether-Fano inequalities. Moreover, using the weighted valuation $\nu_{f}$ the inequalities become equalities.

\begin{thm} \label{NFEqualities}
Suppose that $f:X \dashto Y$ is as above. Consider $D$ an effective $\mathbb{Q}$-Weil divisor on $X$ with $D \sim_{\mathbb{Q}} -mK_{X}$. Then 
\begin{enumerate}
\item $\nu_{f}\left(\frac{1}{m}D\right) = \sdeg(f^{-1}) = 1 + a(\nu_{f})$ if and only if $f_{\ast}D \sim_{q,\mathbb{Q}} 0$;
\item $\nu_{f}\left(\frac{1}{m}D\right) = \sdeg(f^{-1}) - \frac{1}{\sdeg(f)} = 1 + a(\nu_{f}) - \frac{1}{\sdeg(f)}$ for $D$ the strict transform under $f^{-1}$ of a general element of a base point free linear system $\mathcal{H} \sim_{\mathbb{Q}} -\ell K_{Y}$ (if such a linear system exists).  
\end{enumerate}
\end{thm}
\begin{proof}
\indent For the first equality we note that by Proposition \ref{Pushforward}, $f_{\ast}D \sim_{q,\mathbb{Q}} 0$ exactly when $\nu_{f}(D/m) = \sdeg(f^{-1}) = 1 + a(\nu_{f})$. \\
\indent For the second equality we apply Proposition \ref{Pushforward} to both $f$ and $f^{-1}$. Then by pushing forward twice we conclude
\[1 = \sdeg(f)\cdot (\sdeg(f^{-1})-\nu_{f}(D/m)).\]
Solving for $\nu_{f}(D/m)$ gives the desired equality.
\end{proof}

From Theorem \ref{NFEqualities}, we can also make conclusions about the canonical thresholds of the divisors considered in the above proposition. To make sense of the canonical threshold we at the very least need to start with a variety with at worst canonical singularities. We first recall the definition of canonical threshold. 

\begin{defn}
Consider a normal variety $X$ with canonical singularities and let $D$ be a Weil $\mathbb{Q}$-divisor on $X$. Then the \textit{canonical threshold} is the number
\[\ct(X,D) = \sup\{t \in \mathbb{R}| (X,cD) \text{ is canonical}\}.\]
\end{defn}
 
 The following allows us to give an upper bound on canonical thresholds in terms of the degrees of $f$ and $f^{-1}$.

\begin{prop}\label{GeneralNFInequalities}
Suppose that $f:X \dashto Y$ is as above where $X$ has canonical singularities and all $e_{i} \geq 0$. Consider $D \sim_{\mathbb{Q}} -mK_{X}$ a $\mathbb{Q}$-Weil divisor on $X$. Then we have the following bounds.
\begin{enumerate}
    \item If $f_{\ast}D \sim_{q,\mathbb{Q}} 0 $ then 
    \[\ct\left(X,\frac{1}{m}D\right) \leq \frac{\sdeg(f^{-1})-1}{\sdeg(f^{-1})}.\]
    \item If $D = f_{\ast}^{-1}H$ where $H \sim_{\mathbb{Q}} -mK_{Y}$ is a general element of a base point free linear system, then 
    \[\ct\left(X,\frac{1}{m}D\right) \leq \frac{\sdeg(f^{-1})-1}{\sdeg(f^{-1}) - \frac{1}{\sdeg(f)}}\]
\end{enumerate}
\end{prop}
\begin{proof}
This follows by applying the  following lemma using the fact that $a(\nu_{f}) = \sdeg(f^{-1}) - 1$ and Theorem \ref{NFEqualities}.  
\end{proof}

\begin{lem}
Suppose that $\nu = \sum_{i}e_{i}\ord_{E_{i}}$ is a weighted valuation on a $\mathbb{Q}$-factorial variety $X$ with all $e_{i} \geq 0$. Then for any divisor $D$ on $X$
\[\ct(X,D) \leq \frac{a(\nu)}{\nu(D)}\]
\end{lem}
\begin{proof}
We show that $a(E_{i},X)/\ord_{E_{i}}(D) \leq a(\nu)/\nu(D) =: c$ for at least one of the $E_{i}$. If this were not the case then $a(E_{i},X) > c\ord_{E_{i}}(D)$ for all $E_{i}$. Hence 
\[c = \frac{\sum e_{i}a(E_{i},X)}{\sum_{i}e_{i}\ord_{E_{i}}(D)} > \frac{c\sum e_{i}\ord_{E_{i}}(D)}{\sum_{i}e_{i}\ord_{E_{i}}(D)} = c\]
Hence we can conclude that $\ct(X,D) \leq \frac{a(\nu)}{\nu(D)}$. 
\end{proof}

\indent We can now prove Theorem \ref{NFInequalities}. 

\begin{proof}[Proof of Theorem \ref{NFInequalities}]
We apply Proposition \ref{GeneralNFInequalities}. As $X$ has Picard number 1, every divisor is $\mathbb{Q}$-linearly equivalent to some multiple of $-K_{X}$. Moreover  by Proposition \ref{degbound} $d' > 1$. Finally all the $e_{i}$ are positive as $Y \to S$ is a Mori fiber space. Therefore the Proposition applies and we get the desired inequalities on the canonical thresholds. 
\end{proof}

\section{Constructing Links}

\indent The Noether-Fano inequalities allow us to construct singular divisors from a birational map $f:X \dashto Y$. Moreover, using \ref{NFEqualities} we can describe exactly how singular the divisors are using the weighted valuation $\nu_{f}$. When there is only a single divisor contracted by $f^{-1}$, this weighted valuation is an actual valuation. In that case, we can ask if the inequalities on canonical thresholds are actually equalities. \\
\indent We show that starting with $D \sim -mK_{X}$ a divisor or movable linear system on $X$ such that $\left(X,\frac{1}{m}D\right)$ is not canonical we can find a particular birational map $f:X \dashto Y$ realizing the canonical threshold. The techniques in this section are well known and are in fact how the first Sarkisov links out of $X$ are constructed in \cite{corti1995}.

\begin{prop}\label{blowup}
Suppose that $D$ is either a $\mathbb{Q}$-divisor or a non-free  movable $\mathbb{Q}$-linear system on $X$ with $c := \ct\left(X,D\right)$. Then there exists a divisorial contraction $\phi:\tilde{X} \to X$ such that 
$\tilde{X}$ is $\mathbb{Q}$-factorial and terminal, the exceptional divisor $E$ computes the canonical threshold $c$ of $(X,D)$, and the pair $(\tilde{X},cD_{\tilde{X}})$ is canonical, where $D_{\tilde{X}}$ denotes the strict transform of $D$ on $\tilde{X}$.
\end{prop}
\begin{proof}
First we consider a log resolution $X_{1} \to X$ of the pair $(X,cD)$. By \cite{bchm} we can find a birational morphism $X_{2} \to X$ whose exceptional divisors are exactly the exceptional divisors $E_{i}$ of $X_{1} \to X$ such that $a(E_{i},X,c\Delta) = 0$. The variety $X_{2}$ will be $\mathbb{Q}$-factorial and $(X_{2},cD_{X_{2}})$ will be terminal where $D_{2}$ denotes the strict transform of $D$
on $X_{2}$.\\
\indent To obtain a divisorial contraction to $X$ we run an ordinary $K_{X_{2}}$-MMP on $X_{2}$ relative to $X$. The last step of this MMP will be a divisorial contraction $\phi:\tilde{X} \to X$ with the desired properties. In particular $\tilde{X}$ is guaranteed to be terminal.
\end{proof}

\begin{prop}\label{LinkConstruction}
Suppose that $D \sim_{\mathbb{Q}} - mK_{X}$ is either a $\mathbb{Q}$-divisor or a non-free  movable $\mathbb{Q}$-linear system on $X$ a terminal Fano variety of Picard number 1. Let $c := \ct\left(X,D\right)$, and assume that $c < 1/m$. Then there exists one of the following  diagrams.
\begin{center}
\begin{tikzcd}
\tilde{X} \arrow[r,dashed] \arrow[d] & Y \arrow[d] & & \tilde{X} \arrow[r,dashed] \arrow[d] & \tilde{Y} \arrow[d] \\
X & S & & X & Y
\end{tikzcd}
\end{center}
Here $\tilde{X} \to X$ is a blowup, as constructed in \ref{blowup}, with $\rho(\tilde{X}) = 2$. The horizontal arrow $\tilde{X} \dashto Y$ or $\tilde{X} \dashto \tilde{Y}$ will be a sequence of flips. In the left diagram, termed type I, $Y \to S$ is a terminal Mori fiber space with $\dim(S) \geq 1$ and $\rho(Y) = 2$. In the right diagram, termed type II, $\tilde{Y} \to Y$ is a divisorial contraction to $Y$ a klt Fano variety of Picard number 1. Moreover if $D$ was a movable linear system, then $Y$ has terminal singularities. 
\end{prop}
\begin{proof}
Using Proposition \ref{blowup} take a blowup $\phi:\tilde{X} \to X$ using the pair $(X,cD)$. Note that the pair $-(K_{X} + cD)$ is ample by the assumption that $c < 1/m$. Now let $D_{\tilde{X}}$ denote the strict transform of $D$ on $\tilde{X}$. Then $(\tilde{X},cD_{\tilde{X}})$ is canonical and the pair $-(K_{\tilde{X}} + cD_{\tilde{X}}) \sim_{\mathbb{Q}} -\phi^{\ast}(K_{X} + cD)$ is big and nef. Now by \cite{bchm} we can run a $K_{\tilde{X}} + cD_{\tilde{X}}$-MMP, which will end in a Mori Fiber space $Y \to S$. \\
\indent As $\rho(\tilde{X}) = 2$, we see that either $\rho(Y) = 1$ or $\rho(Y) = 2$. If $\rho(Y) = 1$, the last step of the MMP will be a divisorial contraction $\tilde{Y} \to Y$ and $\dim(S) = 0$. In this case we get a link of type II. Moreover the map $X \dashto Y$ is not an isomorphism as if it were the exceptional divisor $E$ of $\phi$ would have to be contracted, so $\tilde{Y} = \tilde{X}$ as $-E$ is $\phi$-ample. On the other hand, $K_{\tilde{X}} + cD_{\tilde{X}}$ is $\phi$-trivial, so that $\phi$ will not be a $K_{\tilde{X}} + cD_{\tilde{X}}$ contraction. If $\rho(Y) = 2$ then we end up with a type I link. \\
\indent In the case that we have a link of type I, the variety $Y$ will be terminal. If we end up with a type II link then $\tilde{Y}$ will be terminal and $Y$ will be klt. Moreover if no component of $D$ is contracted by $\tilde{Y} \to Y$ then $Y$ is terminal as well. In particular if $D$ is a movable linear system then we get that $Y$ must be terminal. 
\end{proof}

\indent When all the varieties in the above diagram are terminal, these diagrams are called Sarkisov links of type I and II respectively. Any birational map between two Mori fiber spaces factors into a composition of Sarkisov links. The types of links occurring here are the only options for a link starting with a Fano variety of Picard number one. \\
\indent When $D$ is a divisor we have a chance to contract components and end up with a non-terminal singularity. We make the following definition so we can refer to this class of birational maps in the future.

\begin{defn}\label{Slinks}
A link of type II above where $Y$ is klt will be called a \textit{Sarkisov-like link of type II} with $Y$ klt.
\end{defn}

 The following proposition sums up the links we have constructed in terms of the weighted valuations $\nu_{f}$. 
 
 \begin{prop}\label{computecanonicalthresholds}
 Consider $X$ a terminal Fano variety of Picard number 1. 
 \begin{enumerate}
     \item Suppose that $D \sim_{\mathbb{Q}} -mK_{X}$ is a divisor with $c := \ct(X,D) < \frac{1}{m}$. Then there exists a birational map $f:X \dashto Y$ which is either a Sarkisov link of type I, a Sarkiov link of type II, or a Sarkisov type of type II with $Y$ klt. Moreover $c = \frac{a(\nu_{f})}{\nu_{f}(D)}$.
     \item Suppose that $\mathcal{M} \sim_{\mathbb{Q}} -mK_{X}$ is a movable linear system with $c := \ct(X,\mathcal{M}) < \frac{1}{m}$. Then there exists a birational map $f:X \dashto Y$ which is either a Sarkisov link of type I or a Sarkiov link of type II. Moreover $c = \frac{a(\nu_{f})}{\nu_{f}(\mathcal{M})}$.
 \end{enumerate}
 \end{prop}
 
\begin{proof}
We apply the construction of Proposition \ref{LinkConstruction} to the pair $(X,cD)$ or $(X,c\mathcal{M})$ where $c$ denotes the canonical threshold of the pair. This gives the desired birational map $f:X\dashto Y$ to a Mori fiber space $q:Y \to S$. By construction the exceptional divisor $E$ of $\tilde{X} \to X$ computes the canonical threshold. Moreover by definition $\nu_{f} = e\cdot \ord_{E}$ and $a(\nu_{f}) = e\cdot a(E,X)$ where the strict transform of $E$ on $Y$, $E_{Y} \sim_{q,\mathbb{Q}} -eK_{Y}$. Hence we see that $\nu_{f}$ computes the canonical threshold. 
\end{proof}

\section{Global Canonical Thresholds}

In this section we give the proof of Theorem \ref{GlobalCanonicalThresholds} and its corollaries. We first gives the definition of a birationally rigid and superrigid variety as these will be our main source of applications.

\begin{defn}
A terminal Fano variety of Picard number 1 is \textit{birationally rigid} if for any birational map $f:X \dashto Y$ to a terminal Mori fiber space $q:Y \to S$, $Y = X$. $X$ is \textit{birationally superrigid} if moreover any such $f$ is an isomorphism. 
\end{defn}

\indent By \cite{CS08} this is equivalent to $\mct(X) \geq 1$ though Theorem \ref{NFInequalities} combined with Proposition \ref{computecanonicalthresholds}, give a proof as well. We now prove the general theorem on $\ct(X)$ and $\mct(X)$ assuming they are less than one. 

\begin{proof}[Proof of Theorem \ref{GlobalCanonicalThresholds}]
\indent First, given any birational map $f:X \dashto Y$, with $f \in \mathscr{M}$, take $\mathcal{M}_{Y}$ the pullback of a base point free linear system on $S$ by the Mori fiber space structure $q:Y \to S$. Denote by $\mathcal{M}_{X} = f_{\ast}^{-1}\mathcal{M}_{Y} \sim_{\mathbb{Q}} -mK_{X}$, the strict transform on $X$. Then by Theorem \ref{NFInequalities} 
\[\ct\left(X, \frac{1}{m}\mathcal{M}_{X}\right) \leq \frac{d'-1}{d'},\]
where $d' = \sdeg(f^{-1})$. Hence,
\[\ct(X),\mct(X) \leq \inf\left\{ \frac{d'-1}{d'}|f\in \mathscr{M}\right\}.\]
\indent Next we  take $f:X \dashto Y$ to be a map in $\mathscr{F}$. Taking a divisor $D \sim_{\mathbb{Q}} -mK_{X}$ on $X$ contracted by $f$, by Theorem \ref{NFInequalities}, we conclude that
\[\ct\left(X, \frac{1}{m}\mathcal{M}_{X}\right) \leq \frac{d'-1}{d'},\]
where $d' = \sdeg(f^{-1})$. Therefore, 
\[\ct(X) \leq \inf\left\{ \frac{d'-1}{d'}|f\in \mathscr{F}\right\}.\]
\indent Finally we can take $\mathcal{H}_{Y} \sim -\ell K_{Y}$ a base-point-free linear system on $Y$ as $-K_{Y}$ is ample. Letting $\mathcal{H}_{X} := f_{\ast}^{-1}\mathcal{H}_{Y} \sim -mK_{X}$ we see by \ref{NFInequalities}  that 
\[\ct\left(X, \frac{1}{m}\mathcal{H}_{X}\right) \leq \frac{d'-1}{d'-1/d},\]
where $d = \sdeg(f)$. Therefore, 
\[\mct(X) \leq \inf\left\{ \frac{d'-1}{d'-1/d}|f\in \mathscr{F}\right\}.\]
\indent Note that since $\mathscr{S}_{1} \subseteq \mathscr{M}$ and $\mathscr{S}_{2},\mathscr{S}_{2}^{+} \subseteq \mathscr{F}$ it is now clear that 
\[\ct(X) \leq  \inf\left\{\frac{d'-1}{d'}|f \in \mathscr{M},\mathscr{F}\right\} \leq  \inf\left\{\frac{d'-1}{d'}|f \in \mathscr{S}_{1},\mathscr{S}_{2},\mathscr{S}_{2}^{+}\right\}\]
and 

\begin{align*}
    \mct(X) & \leq\inf\left\{\frac{d'-1}{d'}|f \in \mathscr{M}\right\}\cup \left\{\frac{d'-1}{d'-1/d}|f \in \mathscr{F}\right\} \\
&\leq  \inf\left\{\frac{d'-1}{d'}|f \in \mathscr{S}_{1}\right\}\cup \left\{\frac{d'-1}{d'-1/d}|f \in \mathscr{S}_{2}\right\}
\end{align*}
We now show that these inequalities are in fact equalities. First take any divisor $D \sim_{\mathbb{Q}} -mK_{X}$ on $X$ with $\ct\left(X,D\right) < \frac{1}{m}$. Then by Proposition \ref{computecanonicalthresholds} there exists a birational map $f:X \dashto Y$ with $f \in \mathscr{S}_{1} \cup \mathscr{S}_{2} \cup \mathscr{S}_{2}^{+}$ such that 
\[\ct\left(X, \frac{1}{m}D\right) = \frac{a(\nu_{f})}{\nu_{f}(D)} \geq \frac{d'
-1}{d'},\]
which proves the equalities for $\ct(X)$. \\
\indent To prove the other set of equalities we take any movable linear $\mathcal{M} \sim_{\mathbb{Q}} -mK_{X}$ on $X$ with $\ct\left(X,\mathcal{M}\right) < \frac{1}{m}$. Then by Proposition \ref{computecanonicalthresholds} there exists a birational map $f:X \dashto Y$ with $f\in \mathscr{S}_{1} \cup \mathscr{S}_{2}$ such that 
\[\ct\left(X, \frac{1}{m}\mathcal{M}\right) = \frac{a(\nu_{f})}{\nu_{f}(\mathcal{M})} \geq \frac{d'
-1}{d'-1/d},\]
which proves the equalities for $\mct(X)$.
\end{proof}

We now prove Corollary \ref{mctrigid}. The idea is that if we compute (or at least can bound) all the degree of Sarkisov links out of a Fano variety $X$ then we can compute (or bound) $\mct(X)$.

\begin{proof}[Proof of Corollary \ref{mctrigid}]
Since $X$ is birationally rigid every Sarkisov link is a Type II link $f:X \dashto X$. Note that as $\Pic(X) = \Cl(X) = \mathbb{Z} \cdot -K_{X}$ we must have $\sdeg(f),\sdeg(f^{-1}) \in \mathbb{Z}$. Moreover as $X$ is terminal, $\sdeg(f^{-1}),\sdeg(f) \geq 2$. Therefore letting $d = \sdeg(f)$ and $d' = \sdeg(f^{-1})$ we conclude $\frac{d'-1}{d'-1/d} > \frac{d'-1}{d'} \geq \frac{1}{2}$. As this holds for all Sarkiov links, by Theorem \ref{GlobalCanonicalThresholds} $\mct(X) > \frac{1}{2}$. \\
\indent If $\Bir(X)$ is generated by involutions of Sarkisov degree $d$, then for such Sarkisov links we have $d = d'$. In particular, $\mct(X) = \frac{d-1}{d-1/d} = \frac{d}{d+1} \geq \frac{2}{3}$.
\end{proof}

\indent For the other corollaries we instead use bounds on $\ct(X)$ and $\mct(X)$ to prove nonexistence of birational maps to other Mori fiber spaces. 

\begin{proof}[Proof of Corollary \ref{ct1}]
Consider a birational map $f:X \dashto Y$ where $Y$ is a Fano variety of Picard number $1$ (with arbitrarily bad singularities). As both varieties are $\mathbb{Q}$-factorial and have Picard number $1$,  $f$ must contract some divisor $F$ on $X$. In particular as $X$ is terminal $d'= \sdeg(f^{-1}) > 1$ by Proposition  \ref{degformula}. This implies $\ct(X) \leq \frac{d'-1}{d'} < 1$ which contradicts the assumption that $\ct(X) = 1$.
\end{proof}

\begin{proof}[Proof of Corollary \ref{bsrmaps}]
Consider a birational map $f:X \dashto Y $ where $Y$ is a Fano variety together with a base-point-free linear system $\mathcal{H}_{Y} \sim -\ell K_{X}$ with $\ell <  1$. Note that the strict transform $f_{\ast}^{-1}\mathcal{H}_{X} \sim -\sdeg(f)\cdot \ell K_{X}$. Therefore as $\Cl(X) = \Pic(X) = \mathbb{Z} \cdot -K_{X}$,  $\sdeg(f) \cdot \ell$ must be an integer $\geq 1$. As $\ell < 1$, we  conclude $\sdeg(f) > 1$. Moreover as $X$ is terminal $\sdeg(f^{-1}) > 1$. In particular letting $d= \sdeg(f)$ and $d' = \sdeg(f^{-1})$, 
\[\mct(X) \leq  \frac{d'-1}{d'-1/d}  < 1\]
This contradicts the fact that $X$ is birationally superrigid. 
\end{proof}

\bibliographystyle{plain}
\bibliography{NFbib}

\begin{thebibliography}{10}

\bibitem{birkar2016singularities}
Caucher Birkar.
\newblock Singularities of linear systems and boundedness of fano varieties.
\newblock {\em arXiv preprint arXiv:1609.05543}, 2016.

\bibitem{bchm}
Caucher Birkar, Paolo Cascini, Christopher~D Hacon, and James McKernan.
\newblock Existence of minimal models for varieties of log general type.
\newblock {\em Journal of the American Mathematical Society}, 23(2):405--468,
  2010.

\bibitem{CS08}
Ivan~A Cheltsov and Konstantin~A Shramov.
\newblock Log canonical thresholds of smooth fano threefolds.
\newblock {\em Russian Mathematical Surveys}, 63(5):859, 2008.

\bibitem{corti1995}
Alessio Corti.
\newblock Factoring birational maps of 3-folds after sarkisov.
\newblock {\em J. Algebraic Geom.}, 4:223--254, 1995.

\bibitem{df2013}
Tommaso de~Fernex.
\newblock Birationally rigid hypersurfaces.
\newblock {\em Inventiones mathematicae}, 192(3):533--566, 2013.

\bibitem{hmsarkisov}
Christopher~D Hacon and James McKernan.
\newblock The sarkisov program.
\newblock {\em arXiv preprint arXiv:0905.0946}, 2009.

\bibitem{IM71}
Vasily~A Iskovskih and Ju~I Manin.
\newblock Three-dimensional quartics and counterexamples to the l{\"u}roth
  problem.
\newblock {\em Mathematics of the USSR-Sbornik}, 15(1):141, 1971.

\bibitem{iskovskikh2004noether}
Vasilij~A Iskovskikh.
\newblock On the noether--fano inequalities.
\newblock {\em arXiv preprint math/0412523}, 2004.

\bibitem{pukhlikov2018canonical}
Aleksandr~V Pukhlikov.
\newblock Canonical and log canonical thresholds of fano complete
  intersections.
\newblock {\em European Journal of Mathematics}, 4(1):381--398, 2018.

\bibitem{pukhlikovindex2}
Aleksandr~V Pukhlikov.
\newblock Birational geometry of singular fano hypersurfaces of index two.
\newblock {\em manuscripta mathematica}, 161(1):161--203, 2020.

\bibitem{stibitzzhuang}
Charlie Stibitz and Ziquan Zhuang.
\newblock K-stability of birationally superrigid fano varieties.
\newblock {\em arXiv preprint arXiv:1802.08381}, 2018.

\bibitem{tian87}
Gang Tian.
\newblock On k{\"a}hler-einstein metrics on certain k{\"a}hler manifolds with
  $c_1(m)> 0$.
\newblock {\em Inventiones mathematicae}, 89(2):225--246, 1987.

\bibitem{zhu2020higher}
Ziwen Zhu.
\newblock Higher codimensional alpha invariants and characterization of
  projective spaces.
\newblock {\em International Journal of Mathematics}, 31(02):2050012, 2020.

\end{thebibliography}

\end{document}